\documentclass[a4paper,reqno]{amsart} 

\usepackage{amssymb}
\usepackage[mathcal]{euscript} 

\newcommand{\bydef}{:=}

\newcommand{\Skew}{\mathrm{Skew}}
\newcommand{\id}{\mathrm{id}}

\newcommand{\trace}{\mathrm{tr}}
\newcommand{\sr}{\mathrm{sr}}
\DeclareMathOperator{\charac}{\mathrm{char}}
\newcommand{\norm}{\mathrm{n}}

\newcommand{\cA}{\mathcal{A}}
\newcommand{\cB}{\mathcal{B}} 
\newcommand{\cC}{\mathcal{C}}

\newcommand{\cH}{\mathcal{H}}

\newcommand{\cO}{\mathcal{O}}

\newcommand{\cS}{\mathcal{S}}

\newcommand{\ZZ}{\mathbb{Z}}

\newcommand{\CC}{\mathbb{C}}

\newcommand{\FF}{\mathbb{F}} 
\newcommand{\KK}{\mathbb{K}}

\DeclareMathOperator{\Aut}{\mathrm{Aut}}

\DeclareMathOperator{\Stab}{\mathrm{Stab}}

\DeclareMathOperator{\supp}{\mathrm{Supp}}

\DeclareMathOperator{\Mat}{\mathrm{Mat}}

\newcommand{\frsl}{{\mathfrak{sl}}}

\newcommand{\GL}{\mathrm{GL}} 
\newcommand{\SL}{\mathrm{SL}}
\newcommand{\PGL}{\mathrm{PGL}}
\newcommand{\PSL}{\mathrm{PSL}}

\newcommand{\U}{\mathrm{U}}
\newcommand{\PU}{\mathrm{PU}}

\newcommand{\PSU}{\mathrm{PSU}}

\providecommand{\espan}[1]{\text{span}\left\{ #1\right\}}

\newenvironment{romanenumerate} 
{\begin{enumerate}

}{\end{enumerate}}

\newtheorem{theorem}{Theorem}[section]

\newtheorem{lemma}[theorem]{Lemma}
\newtheorem{corollary}[theorem]{Corollary}

\theoremstyle{definition} 
\newtheorem{definition}[theorem]{Definition}

\theoremstyle{remark} \newtheorem{remark}[theorem]{Remark}
\numberwithin{equation}{section}

\def\hregleta{\hrule height .5pt}
\def\hreglon{\hrule height 1pt}
\def\vreglon{\vrule height 12pt width1pt depth 4pt}
\def\vregleta{\vrule width .5pt}
\def\hreglonfill{\leaders\hreglon\hfill}

\def\hregletafill{\leaders\hregleta\hfill}

\title{Okubo algebras with isotropic norm}

\author{Alberto Elduque} 
\address{Departamento de
Matem\'{a}ticas e Instituto Universitario de Matem\'aticas y
Aplicaciones, Universidad de Zaragoza, 50009 Zaragoza, Spain}
\email{elduque@unizar.es} 
\thanks{Supported by grant
MTM2017-83506-C2-1-P (AEI/FEDER, UE) and by grant E22\_17R (Gobierno de 
Arag\'on, Grupo de investigaci\'on ``\'Algebra
y Geometr{\'\i}a'').}

\begin{document}

\begin{abstract}
Okubo algebras form an important class of nonunital 
composition algebras of dimension $8$. Contrary to what happens for 
unital composition algebras, 
they are not determined by their multiplicative norms. Okubo algebras 
with isotropic norm are characterized here by the existence of a 
\emph{special} grading. In the split case, and under some restrictions 
on the ground field,
the automorphism group of the most symmetric of these gradings is the 
projective unitary group $\PU(3,2^2)$, whose structure is showcased by 
the grading.
\end{abstract}

\maketitle




\bigskip


Unital composition algebras (also termed Hurwitz algebras) over a field are the 
analogues of the classical algebras or real and complex numbers, 
quaternions, and octonions. In particular, their dimension is restricted to $1$, $2$, $4$ or $8$. The reader may consult \cite[Chapter 2]{ZSSS}, \cite[Chapter VIII]{KMRT}, or \cite{SV00}.

In recent years, a new class of composition 
algebras, the so called \emph{symmetric composition algebras}, better 
suited for the study of the triality phenomenon, have made their 
appearance. These split in two disjoint families: para-Hurwitz algebras and Okubo algebras. (See e.g. \cite[Chapter VIII]{KMRT} or \cite{Eld_comp}.)

Two Hurwitz (or para-Hurwitz) algebras over a field are isomorphic if and only if their norms are isometric. Since the norm of any composition algebra is a Pfister form, this implies in particular that, up to isomorphism, in each dimension $2$, $4$, or $8$, there is a unique Hurwitz (or para-Hurwitz) algebra with isotropic norm. These are the \emph{split Hurwitz (or para-Hurwitz) algebras}. Over a field 
$\FF$, the split Hurwitz algebras are, up to isomorphism, 
$\FF\times\FF$, $\Mat_2(\FF)$, and the split Cayley algebra $\cC_s(\FF)$.

The situation for Okubo algebras is not that neat. Over many fields (see Corollaries \ref{co:not3_isotropic} and \ref{co:3_isotropic}), the norm of any Okubo algebra is isotropic, but this does not imply that the Okubo algebra is the split one.

A characteristic feature of the Okubo algebras with isotropic norm is that they are endowed with a grading by $(\ZZ/3)^2$: 
$\cO=\bigoplus_{0\neq a\in (\ZZ/3)^2}\cO_a$, where the dimension of all (nonzero) homogeneous components is $1$.

\smallskip

The paper is structured as follows. The first section will provide the basic definitions about composition algebras. The split Okubo algebra will be defined in a very precise way. Okubo algebras will then be defined as the twisted forms of the split one. 
Section \ref{se:Okubo}  will review the main known results on Okubo algebras. 

The short Section \ref{se:finite} will be devoted to showing that over any finite field the only Okubo algebra is, up to isomorphism, the split one. This will be instrumental in 
Section \ref{se:special}, where the Okubo algebras with 
isotropic norm will be characterized in terms of the existence 
of some gradings on them. The 
automorphisms groups of these gradings will be computed.  The largest 
group that appears is the projective unitary group $\PU(3,2^2)$. 
The classical group $\PSU(3,2^2)$ is not simple, and its structure will 
become clear by means of the only Okubo algebra over the field of 
four elements. 

\smallskip

All the algebras considered in this paper will be defined over a ground field 
$\FF$.

\section{Composition algebras}\label{se:intro}

A quadratic form $\norm:V\rightarrow \FF$ on a vector space $V$ over a field $\FF$ is said to be \emph{nondegenerate} if so is its polar form:
\[
\norm(x,y)\bydef \norm(x+y)-\norm(x)-\norm(y),
\]
that is, if its radical $V^\perp\bydef\{v\in V:\norm(v,V)=0\}$ is trivial.
Moreover, $\norm$ is said to be \emph{nonsingular} if either it is nondegenerate or it satisfies that the dimension of $V^\perp$ is $1$ and $\norm(V^\perp)\neq 0$. The last possibility only occurs over fields of characteristic $2$.

\begin{definition}\label{df:composition}
A \emph{composition algebra} over a field $\FF$ is a triple $(\cC,\cdot,\norm)$ where
\begin{itemize}
\item $(\cC,\cdot)$ is a nonassociative algebra, and
\item $\norm:\cC\rightarrow\FF$ is a nonsingular quadratic form which is \emph{multiplicative}, that is,
\begin{equation}\label{eq:nxy}
\norm(x\cdot y)=\norm(x)\norm(y)
\end{equation}
for any $x,y\in\cC$.
\end{itemize}
\end{definition}

The unital composition algebras are called \emph{Hurwitz algebras}. These have the property that any element satisfies the degree $2$ equation:
\[
x^{\cdot 2}-\norm(x,1)x+\norm(x)1=0.
\]
The map $x\mapsto \overline{x}=\norm(x,1)1-x$
is both an involution and an isometry, and it satisfies
$\norm(x\cdot y,z)=\norm(y,\overline{x}\cdot z)=
\norm(x,z\cdot\overline{y})$ for all $x,y,z$.

For simplicity, we will refer to the composition algebra $\cC$, instead of $(\cC,\cdot,\norm)$.

The well-known \emph{Generalized Hurwitz Theorem} (see e.g.
\cite[(33.17)]{KMRT}) asserts that the Hurwitz algebras are, up to isomorphism, either $\FF$,  quadratic \'etale $\FF$-algebras, quaternion algebras over $\FF$, or Cayley algebras over $\FF$. 

In particular, the dimension of a Hurwitz algebra is restricted to 
$1$, $2$, $4$, or $8$. Moreover, two Hurwtiz algebras are isomorphic if and only if their norms are isometric.

It turns out that the dimension of any finite-dimensional not necessarily unital composition algebra is restricted too to $1$, $2$, $4$, or $8$, but there are examples of nonunital composition algebras of arbitrary infinite dimension \cite{EP_infinite}.

\medskip

\begin{definition}
A composition algebra $(\cS,*,\norm)$ is said to be a \emph{symmetric
composition algebra} if 
$\norm(x*y,z)=\norm(x,y*z)$ for all $x,y,z\in \cS$.
\end{definition}

The condition on the definition above is equivalent to 
\begin{equation}\label{eq:xyx}
(x*y)*x=x*(y*x)=\norm(x)y
\end{equation} 
for all $x,y\in\cS$.

The dimension of any symmetric composition algebra is finite, and hence restricted to $1,2,4$, or $8$.

\medskip

The first examples of symmetric composition algebras are given as follows. Let $(\cC,\cdot,\norm)$ be a Hurwitz algebra and consider the composition algebra $(\cC,\bullet,\norm)$ with the new product given by
\[
    x\bullet y=\overline{x}\cdot\overline{y}.
\]
Then $\norm(x\bullet y,z)=\norm(\overline{x}\cdot\overline{y},z)=\norm(\overline{x},z\cdot y)=\norm(x,\overline{z\cdot y})=\norm(x,y\bullet z)$,
    for any $x,y,z$, so that $(\cC,\bullet,\norm)$ is a symmetric composition algebra. (Note that $1\bullet x=x\bullet 1=\overline{x}=\norm(x,1)1-x$  for any $x$: $1$ is a \emph{para-unit} of $(\cC,\bullet,\norm)$.)

These symmetric composition algebras are called \emph{para-Hurwitz}. 

It is easy to prove that two para-Hurwitz algebras are isomorphic if and only if so are the corresponding Hurwitz algebras.

\medskip

But apart from para-Hurwitz algebras, there is a new class of 
eight-dimensional symmetric composition algebras with different properties. These are the Okubo algebras.

Let $\Mat_3(\FF)$ be the associative algebra of $3\times 3$-matrices and let 
$\frsl(3,\FF)$ be the corresponding special Lie algebra, consisting  of the zero trace matrices. Any element $a\in\Mat_3(\FF)$ is a root of its characteristic polynomial:
\[
\det(X 1-a)=X^3-\trace(a)X^2+\sr(a)X-\det(a)
\]
where $\trace(a)$ is the trace, $\det(a)$ the determinant, and $\sr(a)$ is a quadratic form 
on the coordinates of $a$ which, if the characteristic is not $2$, equals 
$\frac{1}{2}\left(\trace(a)^2-\trace(a^2)\right)$.
(The same happens for any central simple associative algebra of degree $3$. In this case $\trace$ is the generic trace, $\sr$ a suitable quadratic form, and $\det$ its generic norm.)

The polar form of the quadratic form $\sr$: $\sr(a,b)\bydef \sr(a+b)-\sr(a)-\sr(b)$, satisfies (in 
any characteristic)
\[
\sr(a,b)=\trace(a)\trace(b)-\trace(ab)
\]
for any $a,b\in\Mat_3(\FF)$. If the characteristic of $\FF$ is not $3$, then
the restriction of the quadratic form $\sr$ to $\frsl(3,\FF)=1^\perp$ is nondegenerate. 

Assume for a while that the characteristic of our ground field $\FF$ is not $3$, and that $\FF$ contains a primitive cubic root $\omega$ of $1$.

On $\frsl(3,\FF)$, define a new multiplication and a quadratic form as follows:
\begin{equation}\label{eq:*omega}
\begin{split}
x*y&=\omega xy-\omega^2 yx-\frac{\omega-\omega^2}{3}\trace(xy),
\\[6pt]
\norm(x)&=\sr(x).
\end{split}
\end{equation}
A straightforward computation (see \cite{EM93}) gives
\[
\norm(x*y)=\norm(x)\norm(y),\quad (x*y)*x=x*(y*x)=\norm(x)y
\]
for all $x,y\in\frsl(3,\FF)$ and hence $\bigl(\frsl(3,\FF),*,\norm)$ is an eight-dimensional symmetric composition algebra.

In particular, take $\FF=\CC$, the field of complex numbers, and consider the Pauli matrices:
\begin{equation}\label{eq:Pauli_xy}
x=\begin{pmatrix}
1&0&0\\
0&\omega&0\\
0&0&\omega^2
\end{pmatrix}\,,\qquad
y=\begin{pmatrix}
0&1&0\\
0&0&1\\
1&0&0
\end{pmatrix}\,,
\end{equation}
in $\Mat_3(\CC)$, which satisfy
\begin{equation}\label{eq:Pauli}
x^3=y^3=1,\quad yx=\omega xy.
\end{equation}
For $i,j\in\ZZ/3$, $(i,j)\ne (0,0)$, define
\begin{equation}\label{eq:zij}
z_{i,j}\bydef \frac{\omega^{-ij}}{\omega-\omega^2}x^iy^j.
\end{equation}
Then $\{z_{i,j}: (i,j)\ne (0,0)\}$ is a basis of $\frsl(3,\CC)$.

Using \eqref{eq:Pauli} we obtain:
\begin{equation}\label{eq:split}
z_{i,j}*z_{i',j'}=\begin{cases}
0&\text{if $\Delta=2$ or $i+i'=0=j+j'$,}\\
-z_{i+i',j+j'}&\text{if $\Delta=1$,}\\
z_{i+i',j+j'}&\text{otherwise,}
\end{cases}
\end{equation}
where 
$\Delta=\det\left(\begin{smallmatrix} i&j\\ i'&j'\end{smallmatrix}
\right)$.
Moreover, we get $\norm(z_{i,j})=0$ for any $(i,j)\neq (0,0)$,  and
$\norm(z_{i,j},z_{i',j'})$ is $1$ for $i+i'=0=j+j'$ and $0$ otherwise.

Thus the $\ZZ$-span
\[
\cO_\ZZ=\ZZ\text{-}\espan{z_{i,j}\mid i,j\in\ZZ/3,\ (i,j)\neq(0,0)}
\]
is closed under $*$, and $\norm$ restricts to a nonsingular multiplicative quadratic form on $\cO_\ZZ$.

This allows us to define Okubo algebras over arbitrary fields:

\begin{definition}\label{df:Okubo}
The algebra $\cO_\FF\bydef\cO_\ZZ\otimes_{\ZZ}\FF$, with the induced multiplication and nonsingular quadratic form, is called the \emph{split Okubo algebra} over the field $\FF$.

The twisted forms of $(\cO_\FF,*,\norm)$ are called \emph{Okubo algebras}.
\end{definition}

\begin{remark} This is not the original definition of these algebras given by Okubo \cite{Okubo78} and Okubo and Osborn \cite{OO81b}, but it is equivalent to it.
\end{remark}

Let $\overline{\FF}$ be an algebraic closure of the field $\FF$. Take 
$0\neq \alpha,\beta\in\FF$, and consider cubic roots $\alpha^{1/3}$ and 
$\beta^{1/3}$ in $\overline{\FF}$. Then the elements in 
$\cO\otimes_\ZZ\overline{\FF}$ given by
\[
z_{1,0}\otimes\alpha^{1/3},\quad z_{0,1}\otimes\beta^{1/3}
\]
generate, by multiplication and linear combinations with coefficients in 
$\FF$, a twisted form of the split Okubo algebra, with a basis (over 
$\FF$!) consisting of the elements 
$\tilde z_{i,j}=z_{i,j}\otimes \alpha^{i/3}\beta^{j/3}$, 
for $0\leq i,j\leq 2$, 
$(i,j)\neq (0,0)$. 

This Okubo algebra will be denoted by $\cO_{\alpha,\beta}$, with $\cO_{1,1}$ being the split Okubo algebra. The multiplication table in this basis is given in Figure \ref{fig:table_ab}.

\begin{figure}[h!]\label{fig:table_ab}
\[
\vbox{\offinterlineskip
\halign{\hfil\,$#$\enspace\hfil&#\vreglon
 &\hfil\enspace$#$\enspace\hfil
 &\hfil\enspace$#$\enspace\hfil&#\vregleta
 &\hfil\enspace$#$\enspace\hfil
 &\hfil\enspace$#$\enspace\hfil&#\vregleta
 &\hfil\enspace$#$\enspace\hfil
 &\hfil\enspace$#$\enspace\hfil&#\vregleta
 &\hfil\enspace$#$\enspace\hfil
 &\hfil\enspace$#$\enspace\hfil&#\vreglon\cr
* &\omit\hfil\vrule width 1pt depth 4pt height 10pt
   &\tilde z_{1,0}  &\tilde z_{2,0}&
    &\tilde z_{0,1}&\tilde z_{0,2}&
  &\tilde z_{1,1}&\tilde z_{2,2}&
    &\tilde z_{1,2}&\tilde z_{2,1}&\cr
 \noalign{\hreglon}
 \tilde z_{1,0}&&\tilde z_{2,0}&0&
   &-\tilde z_{1,1}&0&&-\tilde z_{2,1}&0&&0&-\alpha\tilde z_{0,1}&\cr
 \tilde z_{2,0}&&0&\alpha\tilde z_{1,0}&
   &0&-\tilde z_{2,2}&&0&-\alpha\tilde z_{1,2}&
    &-\alpha\tilde z_{0,2}&0&\cr
 \multispan{13}{\hregletafill}\cr
 \tilde z_{0,1}&&0&-\tilde z_{2,1}&
      &\tilde z_{0,2}&0&&0&-\beta\tilde z_{2,0}&
         &0&-\tilde z_{2,2}&\cr
 \tilde z_{0,2}&&-\tilde z_{1,2}&0&&0&\beta\tilde z_{0,1}&
       &-\beta\tilde z_{1,0}&0&&-\beta\tilde z_{1,1}&0&\cr
 \multispan{13}{\hregletafill}\cr
 \tilde z_{1,1}&&0&-\alpha\tilde z_{0,1}&
      &-\tilde z_{1,2}&0&&\tilde z_{2,2}&0&
         &-\beta\tilde z_{2,0}&0&\cr
 \tilde z_{2,2}&&-\alpha\tilde z_{0,2}&0&
      &0&-\beta\tilde z_{2,1}&&0&\alpha\beta\tilde z_{1,1}&
            &0&-\alpha\beta\tilde z_{1,0}&\cr
 \multispan{13}{\hregletafill}\cr
 \tilde z_{1,2}&&-\tilde z_{2,2}&0&
     &-\beta\tilde z_{1,0}&0&&0&-\alpha\beta\tilde z_{0,1}&
            &\beta\tilde z_{2,1}&0&\cr
 \tilde z_{2,1}&&0&-\alpha\tilde z_{1,1}&
       &0&-\beta\tilde z_{2,0}&&-\alpha\tilde z_{0,2}&0&
          &0&\alpha\tilde z_{1,2}&\cr
 \multispan{14}{\hreglonfill}\cr}}
\]
\caption{Multiplication table of $\cO_{\alpha,\beta}$}\label{fig:splitOkubo}
\end{figure}

\bigskip

\section{Classification of Okubo algebras}\label{se:Okubo}

The classification of Okubo algebras has a different flavor depending on the
characteristic of the ground field being $\neq 3$ or $3$. 

\subsection{Characteristic not $3$}\phantom{a}
If $\cA$ is a finite-dimensional associative algebra,  $\cA_0$ will denote the subspace of generic trace $0$ elements of $\cA$.

\begin{theorem}[\cite{EM91,EM93,Eld97}]\label{th:classification_not3}
Let $\FF$ be a field of characteristic not $3$ and let $\omega$ be a primitive cubic root of $1$ in an algebraic closure of $\FF$. 
\begin{romanenumerate}
\item If $\omega\in\FF$, then the Okubo algebras over $\FF$ are, up to 
isomorphism, exactly the algebras $(\cA_0,*,\norm)$, where $\cA$ is a central simple associative algebra over $\FF$ of degree $3$ and $*$ and $\norm$ are given by \eqref{eq:*omega}.

Two Okubo algebras are isomorphic if and only if so are the corresponding central simple associative algebras.

\item If $\omega\not\in\FF$, and $\KK=\FF[\omega]$,  then the Okubo algebras over 
$\FF$ are, up to 
isomorphism, exactly the algebras $\bigl(\Skew(\cA_0,\tau),*,\norm\bigr)$, where 
$\cA$ is a central simple associative algebra over $\KK$ of degree $3$, endowed with a $\KK/\FF$-involution $\tau$ of the second kind, 
$\Skew(\cA_0,\tau)$ is the subspace of generic trace $0$ elements 
$x$ with $\tau(x)=-x$, and $*$ and $\norm$ are given by \eqref{eq:*omega}.

Two Okubo algebras are isomorphic if and only if so are the corresponding pairs 
$(\cA,\tau)$, as $\KK$-algebras with involution.
\end{romanenumerate}
\end{theorem}

\begin{corollary}[{\cite[Corollary 8.9]{CEKT}}]\label{co:autos_not3}
Let $\FF$ be a field of characteristic not $3$ and let $\omega$ be a primitive cubic root of $1$ in an algebraic closure of $\FF$. 
\begin{romanenumerate}
\item 
If $\omega\in\FF$, then the group of automorphisms of the Okubo algebra 
$(\cA_0,*,\norm)$, as in Theorem \ref{th:classification_not3}.(i) is isomorphic to 
$\Aut(\cA)$. 

\item 
If $\omega\not\in\FF$, and $\KK=\FF[\omega]$,  then the group of automorphisms of
the  Okubo algebra $\bigl(\Skew(\cA_0,\tau),*,\norm\bigr)$ as in Theorem
\ref{th:classification_not3}.(ii) is isomorphic to $\Aut(\cA,\tau)$ (the group of
automorphisms of $\cA$ commuting with $\tau$).
\end{romanenumerate}
In both cases, any automorphism of $\cA$ (commuting with $\tau$ in (ii))
acts by restriction to $\cA_0$ or $\Skew(\cA_0)$, respectively.
\end{corollary}

The fact that  any central simple division algebra over a field containing a primitive cubic root of unity is cyclic (see e.g. \cite[p.~286]{Pierce}) has the following consequence:

\begin{corollary}[{\cite[Proposition 7.3]{EM93}}]\label{co:not3_isotropic}
Let $\FF$ be a field of characteristic not $3$ containing a primitive cubic root of $1$. Then the norm of any Okubo algebra over $\FF$ is isotropic.
\end{corollary}

\begin{remark}\label{re:symbol}
Actually, if our field $\FF$ contains a primitive cubic root $\omega$ of $1$,  any central simple algebra $\cA$ of degree $3$ is a \emph{symbol algebra}, that is, it is isomorphic to the algebra 
$\bigl(\alpha,\beta\bigr)_{\FF,\omega}$ for some nonzero scalars 
$\alpha,\beta\in\FF$, defined as the algebra generated by
two elements $x,y$ subject to the relations (compare with
 \eqref{eq:Pauli})
\[
x^3=\alpha,\quad y^3=\beta,\quad yx=\omega xy.
\]
It is clear that then $\cA_0$, with multiplication and norm as in 
\eqref{eq:*omega}, is
isomorphic to the Okubo algebra $\cO_{\alpha,\beta}$.

If our field $\FF$ has characteristic $\neq 3$ but does not contain
a primitive cubic root of $1$, consider the quadratic field extension 
$\KK=\FF[\omega]$  as in 
Theorem \ref{th:classification_not3}. For 
$0\neq \alpha,\beta\in\FF^\times$, the symbol algebra 
$\bigl(\alpha,\beta\bigr)_{\KK,\omega}$ is endowed with a second 
kind involution $\tau$ fixing the generators $x$ and $y$. Then
\[
\tau\left(\frac{\omega^{-ij}}{\omega-\omega^2}x^iy^j\right)=
\frac{\omega^{ij}}{\omega^2-\omega}y^jx^i=
-\frac{\omega^{-ij}}{\omega-\omega^2}x^iy^j,
\]
and it turns out that the Okubo algebra $\Skew(\cA_0,\tau)$, 
with multiplication and norm as in 
\eqref{eq:*omega}, is
isomorphic to the Okubo algebra $\cO_{\alpha,\beta}$.
\end{remark}

Among the Okubo algebras with isotropic norm, the split one is characterized as follows:

\begin{theorem}[{\cite[Theorem 5.9]{Canadian}}]\label{th:not3_split}
An Okubo algebra over a field of characteristic not $3$ is split if and only if its norm is isotropic and it contains a nonzero idempotent ($e*e=e$).
\end{theorem}

\medskip

\subsection{Characteristic $3$}\phantom{a}
In characteristic $3$, the classification does not follow the same lines:

\begin{theorem}[\cite{Eld97}]\label{th:classification_3}
Up to isomorphism, the Okubo algebras over a field $\FF$ of characteristic $3$ are precisely the algebras $\cO_{\alpha,\beta}$ for $\alpha,\beta\in\FF^\times$.
\end{theorem}

\begin{corollary}\label{co:3_isotropic}
The norm of any Okubo algebra over a field of characteristic $3$ is isotropic.
\end{corollary}

In \cite{Eld97}, precise conditions for two Okubo algebras, over a field $\FF$ of characteristic $3$, to be isomorphic are given. In particular, if 
$\FF$ is perfect, the only Okubo algebra, up to isomorphism, is the split one.

Over fields of characteristic $3$ there exist three kinds of (nonzero) idempotents in Okubo algebras (see \cite{Canadian}):
\begin{description}
\item [Quaternionic] if the restriction of the norm to the centralizer of the idempotent has rank $4$.
\item[Quadratic] if the restriction of the norm to the centralizer of the idempotent has rank $2$.
\item[Singular] if the restriction of the norm to the centralizer of the idempotent has rank $1$.
\end{description}

The split Okubo algebra is characterized as follows:

\begin{theorem}[{\cite[Corollary 6.5]{Canadian}}]\label{th:3_split}
Let $\cO$ be an Okubo algebra over a field $\FF$ of characteristic $3$ with norm $\norm$. The following conditions are equivalent:
\begin{romanenumerate}
\item $\cO$ is split.
\item $\norm$ is isotropic and $\cO$ contains a singular idempotent.
\item $\norm$ is isotropic and $\cO$ contains a quaternionic idempotent.
\end{romanenumerate}
\end{theorem}

\bigskip

\section{Okubo algebras over finite fields are split}\label{se:finite}

By Wedderburn's Little Theorem, there are no quaternion division algebras over a finite field, and hence the only four-dimensional Hurwitz algebra is, up to isomorphism, the algebra of $2\times 2$-matrices. In particular, its norm is isotropic. As any Cayley algebra contains quaternion subalgebras, it follows that its norm is isotropic too, and again  the only Hurwitz (or para-Hurwitz) algebras of dimension $4$ or $8$ over a finite field are the split ones.

However, the isometry class of the norm does not determine the isomorphism class of Okubo algebras. Even so, the only Okubo algebra over a finite field is the split one. To prove this, we need to recall the following well-known result:

\begin{lemma}\label{le:hermitian}
Let $\FF$ be a finite field, $\KK/\FF$ a quadratic field extension 
($[\KK:\FF]=2$), $\tau$ the nontrivial $\FF$-automorphism of $\KK$, $U$ a finite-dimensional $\KK$-vector space, and 
$h:U\times U\rightarrow \KK$ a nondegenerate hermitian form: 
$h$ is $\FF$-bilinear, $h(\alpha u,v)=\alpha h(u,v)$, and $h(v,u)=\tau\bigl(h(u,v)\bigr)$ for any $u,v\in U$ and $\alpha\in \KK$.

Then there is a $\KK$-basis $\{u_1,\ldots,u_n\}$ in $U$ with $h(u_i,u_i)=1$ and $h(u_i,u_j)=0$ for $1\leq i\neq j\leq n$.
\end{lemma}
\begin{proof} 
This is well known.   If $\FF$ is the field of $q$ elements, then 
$\tau(\alpha)=\alpha^q$ for any $\alpha$. Let $\mu$ be a generator 
of the cyclic group $\KK^\times$. Hence $\FF^\times$ is generated by 
$\mu^{q+1}$, so for any $\alpha\in\FF$ there is an element 
$\beta\in\KK$ such that $\alpha=\beta^{q+1}$.

It cannot be the case that $h(u,u)$ equals $0$ for all $u$ because this would imply 
$h(u,v)=-h(v,u)=-\tau\left(h(u,v)\right)$, and this would give 
$\tau(\alpha)=-\alpha$ for any $\alpha\in\KK$, a contradiction, so let us take an element $u\in U$ with $h(u,u)=\alpha\neq 0$. Pick 
$\beta\in\KK$ with $\beta^{q+1}=\alpha$, then:
\[
h\bigl(\beta^{-1}u,\beta^{-1}u\bigr)
=\beta^{-1}\bigl(\beta^{-1}\bigr)^q h(u,u)
=\alpha^{-1}h(u,u)=1.
\]
Now the argument can be repeated with the restriction 
$h\vert_{(\KK u)^\perp}$.
\end{proof}

\begin{theorem}\label{th:finite}
Let $\cO$ be an Okubo algebra over a finite field $\FF$. Then $\cO$ is split.
\end{theorem}
\begin{proof}
If the characteristic is $3$, this follows from the isomorphism conditions in \cite{Eld97}, which shows in particular that the only Okubo algebra over a perfect field is the split one.

If the characteristic of $\FF$ is not $3$ and $\FF$ contains a primitive cubic root of $1$, then by Wedderburn's Little Theorem the only central simple associative algebra over $\FF$ is $\Mat_3(\FF)$, and hence Theorem \ref{th:classification_not3} shows that there is, up to isomorphism, a unique Okubo algebra over $\FF$.

If the characteristic of $\FF$ is not $3$ and $\FF$ does not contain a primitive cubic root of $1$, let $\KK=\FF[\omega]$ be the field obtained by adjoining one such root to $\FF$. Again, the only central simple associative $\KK$-algebra is $\Mat_3(\KK)$, and any 
$\KK/\FF$-involution of the second kind in $\Mat_3(\KK)$ is given by the adjoint relative to a nondegenerate $\KK/\FF$-hermitian form on 
$\KK^3$. Since there is only one such form, up to isometry, by Lemma \ref{le:hermitian}, the result follows.
\end{proof}

\bigskip

\section{Special gradings on Okubo algebras with\\ isotropic norm}\label{se:special}

Let us recall the definition of equivalent and isomorphic gradings (see \cite{EKmon}).

\begin{definition}\label{df:equiv_isom}
Let $\Gamma:\cA=\bigoplus_{g\in G}\cA_g$ and $\Gamma':\cB=\bigoplus_{h\in H}\cB_h$ be two gradings by the groups 
$G$ and $H$ on
the algebras $\cA$ and $\cB$.
\begin{enumerate}
\item $\Gamma$ and $\Gamma'$ are said to be \emph{equivalent} if there is an algebra isomorphism $\varphi:\cA\rightarrow \cB$ such that for any $g\in G$ in the support $\supp\Gamma$ (this is the set of elements $g\in G$ such that $\cA_g\neq 0$), there is an element $h\in H$ such that $\varphi(\cA_g)=\cB_h$.

In this case we will write $(\cA,\Gamma)\simeq_{\textup{eq}}(\cB,\Gamma')$.

\item $\Gamma$ and $\Gamma'$ are said to be \emph{isomorphic} if
$G=H$ and there is an algebra isomorphism $\varphi:\cA\rightarrow \cB$ such that $\varphi(\cA_g)=\cB_g$ for all $g\in G$.

\item The group $\Aut(\Gamma)$ is the group of autoequivalences of 
$\Gamma$ (i.e. automorphisms $\varphi$ of 
$\cA$ that permute the homogeneous components of $\Gamma$). Its subgroup 
$\Stab(\Gamma)$ is the group of isomorphisms of $\Gamma$ (i.e., autoequivalences that fix each homogeneous component). The quotient $W(\Gamma)=\Aut(\Gamma)/\Stab(\Gamma)$ is called the \emph{Weyl group} of $\Gamma$.
\end{enumerate}
\end{definition}

It must be remarked that, in dealing with gradings by groups on symmetric composition algebras, it is enough to consider gradings by abelian groups 
(see \cite[Proposition 4.49]{EKmon}).

Following the notation used by Hesselink for Lie algebras
 \cite{Hesselink}, a grading by an abelian group $\Gamma:\cA=\bigoplus_{g\in G}\cA_g$ on an algebra $\cA$ will be called \emph{special} if the homogeneous component corresponding to the neutral element is trivial: $\cA_e=0$.

Okubo algebras with isotropic norm are characterized by the existence of special gradings on them.

\begin{theorem}\label{th:isotropic}
Let $(\cO,*,\norm)$ be an Okubo algebra over a field $\FF$. The following conditions are equivalent:
\begin{romanenumerate}
\item The norm $\norm$ is isotropic.
\item There are scalars $\alpha,\beta\in\FF^\times$ such that $\cO$ is isomorphic to $\cO_{\alpha,\beta}$.
\item $\cO$ admits a special grading by an abelian group.
\end{romanenumerate}
\end{theorem}
\begin{proof}
The equivalence of (i) and (ii) follows from \cite[Theorem 4]{Twisted} and Remark \ref{re:symbol} in case the characteristic of $\FF$ is not $3$, and from Theorem \ref{th:classification_3} in characteristic $3$.

It is clear that $\cO_{\alpha,\beta}$ is $\left(\ZZ/3\right)^2$-graded, with 
\begin{equation}\label{eq:Gab}
\deg(\tilde z_{i,j})=(i,j)\ \text{(modulo $3$).}
\end{equation}
This is a special grading, so (ii) implies (iii).

Conversely, according to \cite[Theorem 4.4]{GradingSym}, the only special gradings on Okubo algebras are, up to equivalence, the above 
$\left(\ZZ/3\right)^2$-gradings on the Okubo algebras 
$\cO_{\alpha,\beta}$.
\end{proof} 

We will denote by $\Gamma_{\alpha,\beta}$ the 
$\left(\ZZ/3\right)^2$-grading on 
$\cO_{\alpha,\beta}$ defined in \eqref{eq:Gab}.

\begin{corollary}\label{co:special}
Any special grading $\Gamma$ on an Okubo algebra $\cO$ is equivalent to the canonical grading $\Gamma_{\alpha,\beta}$ on 
$\cO_{\alpha,\beta}$ for some $0\neq \alpha,\beta\in\FF$:
\[
\bigl(\cO,\Gamma\bigr)\simeq_{\textup{eq}} 
\bigl(\cO_{\alpha,\beta},\Gamma_{\alpha,\beta}\bigr).
\]
\end{corollary}

The proof of the above Theorem is based on 
\cite[Theorem 4.4]{GradingSym}, which relies on the next technical result, 
that has its own independent interest.

\begin{lemma}[{\cite[Theorem 3.12 and Corollary 3.17]{GradingSym}}] \label{le:technical}
Let $(\cS,*,\norm)$ be an eight-dimensional symmetric composition algebra over a field $\FF$, and let $x,y\in\cS$ be elements such that
\[
\norm(x)=\norm(y)=0,\ 
\norm(x,x*x)=\alpha\neq 0\neq \beta=\norm(y,y*y),
\]
and
\[
\norm(\FF x+\FF x*x,\FF y+\FF y*y)=0.
\]
Then $(\cS,*,\norm)$ is an Okubo algebra and either $x*y=0$ or $y*x=0$, but not both.

Moreover, if $y*x=0$, then $(\cS,*,\norm)$ is isomorphic to $\cO_{\alpha,\beta}$ under an isomorphism that takes $x$ to $\tilde z_{1,0}$ and $y$ to $\tilde z_{0,1}$.
\end{lemma}

In the above Lemma, if $x*y=0$, interchanging the roles of $x$ and $y$, it follows that $(\cS,*,\norm)$ is isomorphic to $\cO_{\beta,\alpha}$ under an isomorphism that takes $y$ to $\tilde z_{1,0}$ and $x$ to $\tilde z_{0,1}$.

\medskip

Let $\Gamma:\cO=\bigoplus_{g\in G}\cO_g$ be a special grading
on an Okubo algebra. Denote by $H$ the subgroup of $G$ generated by the support $\supp\Gamma=\{g\in G\mid \cO_g\neq 0\}$. Consider the following map:
\[
\begin{split}
\Phi_\Gamma: H&\longrightarrow \FF^\times/(\FF^\times)^3\\
 h&\mapsto \begin{cases} [1]&\text{if $h=e$,}\\
          [\norm(u,u*u)]&\text{for $0\neq u\in\cO_h$, otherwise.}
\end{cases}
\end{split}
\]
(Here $[\alpha]\bydef (\alpha\FF^\times)^3$ for all $\alpha\in\FF^\times$.)

\begin{lemma}\label{le:Phi_Gamma}
The above map $\Phi_\Gamma$ is a well defined group homomorphism.
\end{lemma}
\begin{proof}
By \cite[Theorem 4.4]{GradingSym}, $H=\{e\}\cup \supp\Gamma$ (disjoint union) is isomorphic to $\left(\ZZ/3\right)^2$. Besides, for any $e\neq h\in H$ and $0\neq u\in\cO_h$, $0\neq u*u\in\cO_{h^2}$ and 
$\norm(u,u*u)\neq 0$ because 
$\norm(\cO_g,\cO_h)=0$ unless $gh=e$ (and $\norm(\cO_g)=0$ unless $g^2=e$, \cite[\S 4]{GradingSym}).
Therefore, $\Phi_\Gamma$ is well defined.

For $0\neq u\in\cO_h$, \eqref{eq:xyx} gives
\[
\begin{split}
(u*u)*(u*u)&=-\bigl((u*u)*u\bigr)*u+\norm(u,u*u)u\\
 &=-\norm(u)u+\norm(u,u*u)u\\
&=\norm(u,u*u)u,
\end{split}
\]
as $\norm(\cO_h)=0$. Hence we get
\[
\norm\bigl(u*u,(u*u)*(u*u)\bigr)=\norm(u,u*u)^2
\]
so that $\Phi_\Gamma(h^2)=\Phi_\Gamma(h)^2$.

Take now $e\neq h,g\in H$ with $g\neq h,h^2$, and pick nonzero elements $x\in\cO_h$ and $y\in\cO_g$. By 
Lemma \ref{le:technical}, either $x*y=0$ or $y*x=0$, but not both.

If $x*y=0$ (the other case is similar), then using repeatedly \eqref{eq:xyx} and the fact that two homogeneous spaces are orthogonal relative to $\norm$ unless the product of their degrees is $e$, we get $(y*y)*x=-(x*y)*y=0$ and
$(y*x)*(y*x)=-x*(y*(y*x))=x*(x*(y*y))=-(y*y)*(x*x)$, so that
\begin{multline}\label{eq:nyx}
\norm\bigl(y*x,(y*x)*(y*x)\bigr)=-\norm\bigl(y*x,(y*y)*(x*x)\bigr)\\
=-\norm\bigl((y*x)*(y*y),x*x\bigr)=-\norm(y,y*y)\norm(x,x*x).
\end{multline}
This shows $\Phi_\Gamma(gh)=\Phi_\Gamma(g)\Phi_\Gamma(h)$. 
\end{proof}

Note that the image of the homomorphism $\Phi_\Gamma$ is a
$3$-elementary abelian group of rank $\leq 2$.

Denote by $\mu_3(\FF)$ the group of cubic roots of $1$ in the field 
$\FF$. This is either trivial or cyclic of order $3$ (isomorphic to 
$\ZZ/3$).

\medskip

Among the finite classical groups it is well known that the projective 
special linear groups $\PSL(n,p^m)$ (for a prime number $p$) are all 
simple with the exceptions of $\PSL(2,2)$ and $\PSL(2,3)$.
 Their action on the projective line show easily that $\PSL(2,2)$ is, up to isomorphism, 
 the symmetric group of degree $3$
  and $\PSL(2,3)$ is the alternating group of degree $4$. Besides, the projective special unitary groups
$\PSU(n,p^{2m})$ are simple with the following exceptions: $\PSU(2,2^2)$,
$\PSU(2,3^2)$, and $\PSU(3,2^2)$. Moreover, $\PSU(2,2^2)$ is 
 the symmetric group of degree $3$ and $\PSU(2,3^2)$ is the alternating group
of order $4$. (See e.g. \cite[Chapter 11]{Grove}.)
 The structure of the remaining case, 
$\PSU(3,2^2)$, turns out to be related to Okubo algebras and their special
gradings.

\begin{theorem}\label{th:special}
Let $\Gamma$ be a special grading on an Okubo algebra $\cO$ over the field $\FF$. Then the short exact sequence
\begin{equation}\label{eq:short}
1\longrightarrow \Stab(\Gamma)\longrightarrow \Aut(\Gamma)
\longrightarrow W(\Gamma)\longrightarrow 1
\end{equation}
splits. 
Moreover, $\Stab(\Gamma)=\mu_3(\FF)^2$, which is trivial if the 
characteristic of $\FF$ is $3$ or if  $\charac\FF\neq 3$  
but $\FF$ does not contain a primitive cubic root of $1$, and it is 
isomorphic to $\left(\ZZ/3\right)^2$ if $\charac\FF\neq 3$  
and $\FF$ contains a primitive cubic root of $1$. The following 
possibilities appear: 
\begin{itemize}
\item The image of $\Phi_\Gamma$ is trivial. In this case $\cO$ is the split Okubo algebra, $(\cO,\Gamma)\simeq_{\textup{eq}}\bigl(\cO_{1,1},\Gamma_{1,1}\bigr)$, and $W(\Gamma)$ is isomorphic to the special linear group $\SL(2,3)$.

\item The image of $\Phi_\Gamma$ is cyclic of order $3$. In this case 
$\bigl(\cO,\Gamma\bigr)\simeq_{\textup{eq}}
\bigl(\cO_{1,\alpha},\Gamma_{1,\alpha}\bigr)$, where $[\alpha]$ is a generator of the image of $\Phi_\Gamma$, and the Weyl group is cyclic of 
order $3$.

\item  The image of $\Phi_\Gamma$ is isomorphic to $\left(\ZZ/3\right)^2$. In this case the Weyl group $W(\Gamma)$ is trivial, 
and if two generators $[\alpha]$ and $[\beta]$ of the image of 
$\Phi_\Gamma$ are picked, then either 
$\bigl(\cO,\Gamma\bigr)\simeq_{\textup{eq}}
\bigl(\cO_{\alpha,\beta},\Gamma_{\alpha,\beta}\bigr)$, or
$\bigl(\cO,\Gamma\bigr)\simeq_{\textup{eq}}
\bigl(\cO_{\beta,\alpha},\Gamma_{\beta,\alpha}\bigr)$, but not both.
\end{itemize}
\end{theorem}

\begin{proof}
The subgroup $H$ generated by the support of $\Gamma$ is isomorphic 
to $\left(\ZZ/3\right)^2$, and the proof of Theorem \ref{th:isotropic} shows that 
$(\cO,\Gamma)\simeq_{\textup{eq}}
(\cO_{\alpha,\beta},\Gamma_{\alpha,\beta})$ for some 
$\alpha,\beta\in \FF^\times$. Then $W(\Gamma)$ can be identified with a subgroup of $\Aut(H)$, and hence by a subgroup of 
$\GL(2,3)=\Aut\left(\bigl(\ZZ/3\bigr)^2\right)$. Equation \eqref{eq:split} and Lemma \ref{le:technical} shows that 
$W(\Gamma)$ is, under the identification above, precisely
\begin{equation}\label{eq:WG}
W(\Gamma)=\{f\in\SL(2,3)\mid \Phi_\Gamma(f.a)=\Phi_\Gamma(a)\ 
\forall a\in\left(\ZZ/3\right)^2\}.
\end{equation}

If the image of $\Phi_\Gamma$ is trivial, and $a,b$ generate $H$, there are elements $x\in\cO_a$, $y\in\cO_b$ with 
$\norm(x,x*x)=1=\norm(y,y*y)$ and either $x*y=0$ or $y*x=0$, 
but not both. If $x*y\neq 0$, 
Lemma \ref{le:technical} shows that there is an isomorphism 
$\cO\rightarrow\cO_{1,1}$ mapping $x$ to $\tilde z_{1,0}$ and $y$ to $\tilde z_{0,1}$, while if $x*y=0$, then \eqref{eq:xyx} shows that 
$(y*y)*x=-(x*y)*y=0$. Lemma \ref{le:technical} gives 
$x*(y*y)\neq 0$.  Replacing $y$ by $y*y$, which is homogeneous too, 
 we get that there is an isomorphism 
$\cO\rightarrow\cO_{1,1}$ mapping $x$ to $\tilde z_{1,0}$ and 
$y*y$ to $\tilde z_{0,1}$. This shows $(\cO,\Gamma)\simeq_{\textup{eq}}\bigl(\cO_{1,1},\Gamma_{1,1}\bigr)$.

Moreover, $W(\Gamma)$ is here the whole $\SL(2,3)$, and \eqref{eq:split} shows that for any $f\in\SL(2,3)$, the linear map that takes $\tilde z_a$ to $\tilde z_{f(a)}$ for any 
$a\in \left(\ZZ/3\right)^2$ is an automorphism. This shows that \eqref{eq:short} splits.

\smallskip

If the image of $\Phi_\Gamma$ is cyclic of order $3$, take
 $a\in\ker\Phi_\Gamma$ and $b\in H\setminus\ker\Phi_\Gamma$ 
with $\Phi_\Gamma(b)=[\alpha]$. By Lemma \ref{le:technical} either 
$\cO_a*\cO_b\neq 0$, or $\cO_b*\cO_a\neq 0$. In the latter case,
as above, we have
$\cO_{a^2}*\cO_b\neq 0$. Thus, changing $a$ by $a^2$ if necessary,
we may assume $\cO_a*\cO_b\neq 0$. Lemma \ref{le:technical} shows 
$(\cO,\Gamma)\simeq_{\textup{eq}}
\bigl(\cO_{1,\alpha},\Gamma_{1,\alpha}\bigr)$. Identify $H$ with 
$\left(\ZZ/3\right)^2$ by means of $a\mapsto (1,0)$ and $b\mapsto (0,1)$. Because of
\eqref{eq:WG}, any $f\in W(\Gamma)$ satisfies $f.a=a$ or $f.a=a^2$,
and $f.b$ is either $b$, $ab$, or $a^2b$.  In other words, the coordinate matrix of $f$ is $\left(\begin{smallmatrix} \pm 1&i\\ 
0&1\end{smallmatrix}\right)$ for some $i\in\ZZ/3$. But $f$ is in 
$\SL(2,3)$, so the first entry must be $1$. Hence $W(\Gamma)$ is cyclic of order $3$. Take $x\in\cO_a$ and $y\in\cO_b$ with 
$\norm(x,x*x)=1$ and $\norm(y,y*y)=\alpha$. As 
$\norm\bigl(x*y,(x*y)*(x*y)\bigr)=-\alpha$ (see \eqref{eq:nyx}),
Lemma \ref{le:technical} shows the existence of an automorphism 
$\varphi$ of $\cO$ with $\varphi(x)=x$ and $\varphi(y)=-x*y$. Using repeatedly \eqref{eq:xyx}, we get
\[
\varphi^3(y)=-\Bigl(x*\bigl(x*(x*y)\bigr)\Bigr)
=(x*y)*(x*x)=\norm(x,x*x)y=y,
\]
so we obtain $\varphi^3=\id$. This shows that \eqref{eq:short}
splits in this case, as $\varphi\in\Aut(\Gamma)$.

\smallskip

Finally, If the image of $\Phi_\Gamma$ is isomorphic to $\left(\ZZ/3\right)^2$, $\Phi_\Gamma$ is one-to-one, and $W(\Gamma)$ is 
trivial because of \eqref{eq:WG}. Let $a,b$ be generators of $H$ with 
$\Phi_\Gamma(a)=[\alpha]$ and $\Phi_\Gamma(b)=[\beta]$, and take
$x\in\cO_a$ with $\norm(x,x*x)=\alpha$ and $y\in\cO_b$ with
$\norm(y,y*y)=\beta$. Lemma \ref{le:technical} shows that if
$\cO_a*\cO_b\neq 0$ (i.e., $x*y\neq 0$), then 
$(\cO,\Gamma)\simeq_{\textup{eq}}
\bigl(\cO_{\alpha,\beta},\Gamma_{\alpha,\beta}\bigr)$ by means of an
isomorphism that takes $x$ to $\tilde z_{1,0}$ and $y$ to 
$\tilde z_{0,1}$, while if $\cO_a*\cO_b=0$, then 
$\cO_b*\cO_a\neq 0$ and then 
$(\cO,\Gamma)\simeq_{\textup{eq}}
\bigl(\cO_{\beta,\alpha},\Gamma_{\beta,\alpha}\bigr)$.
\end{proof}

If $\FF$ is a field with $\charac\FF\neq 3$, 
$\omega$ will denote a primitive cubic root of $1$ in an 
algebraic closure of $\FF$.

As a consequence of Theorem \ref{th:special}, $\Aut(\Gamma)$ is $W(\Gamma)$ if 
$\charac\FF=3$ or 
$\charac\FF\neq 3$ and $\omega\not\in\FF$. Otherwise $\Aut(\Gamma)$ is the
semidirect product $\left(\ZZ/3\right)^2\rtimes W(\Gamma)$, where
$W(\Gamma)$ is identified with a subgroup of $\Aut(\Gamma)$ as shown in
the proof above.

In particular, for the canonical grading $\Gamma_{1,1}$ on the split Okubo algebra
$\cO_{1,1}$ over a field $\FF$ with $\charac\FF\neq 3$ and $\omega\in \FF$,
we obtain:
\begin{equation}\label{eq:AutG11}
\Aut(\Gamma_{1,1})=\left(\ZZ/3\right)^2\rtimes \SL(2,3),
\end{equation}
where the action of $\SL(2,3)$ on $\left(\ZZ/3\right)^2$ is the natural one.

Moreover, this is independent of the field $\FF$. The smallest such field 
is the field of four elements $\FF_4=\{0,1,\omega,\omega^2=1+\omega\}$.
The (unique up to isomorphism by Theorem \ref{th:finite}) 
Okubo algebra $\cO=\cO_{1,1}$ over $\FF_4$ is
the algebra $\bigl(\frsl(3,\FF_4),*,\norm\bigr)$ as in 
Theorem \ref{th:classification_not3}, whose group of automorphisms is $\PGL(3,4)$
(the group of automorphisms of $\Mat_3(\FF_4)$). In other words, the automorphisms
of $\cO$ are obtained by conjugation by elements in $\GL(3,4)$.

Let $\tau$ denote the nontrivial
automorphism of $\FF_4$: $\tau(\alpha)=\alpha^2$ for all 
$\alpha\in\FF_4$. Its
fixed subfield is $\FF_2=\{0,1\}$.
Let $V$ be  the three-dimensional vector space $\FF_4^3$ over $\FF_4$, and endow $V$ with the hermitian form $h:V\times V\rightarrow \FF_4$ where the canonical basis is orthonormal: $h(u_i,u_j)=\delta_{ij}$. (Note that we have 
$h(v,u)=\tau\bigl(h(v,u)\bigr)$ for all $u,v\in V$.) The corresponding group of isometries
is the unitary group $\U(3,2^2)$, which is the subgroup of $\GL(3,4)$ consisting
of those matrices $a$ with $a^\dagger a=1$, where $\dagger$ denotes the 
second kind involution on $\Mat_3(\FF_4)$ determined by $h$: 
$h(a.u,v)=h(u,a^\dagger .v)$ for all $a\in\Mat_3(\FF_4)$ and $u,v\in V$.

Pauli matrices $x,y$ in \eqref{eq:Pauli_xy} are unitary 
($x^\dagger x=1=y^\dagger y$) and their minimal polynomial is $X^3-1$. 
It follows that 
the matrices $z_{ij}$ as in \eqref{eq:zij} are all unitary too and, for $(i,j)\neq (0,0)$, their minimal polynomial is again $X^3-1$. (Note that
in $\FF_4$, $\omega-\omega^2=1$, so $z_{i,j}=\omega^{-ij}x^iy^j$.)

The set of (nonzero) homogeneous elements of $\Gamma_{1,1}$ is 
\[
\cH=\bigcup_{(0,0)\neq (i,j)\in\left(\ZZ/3\right)^2}\FF_4^\times z_{i,j}.
\]

\begin{lemma}\label{le:orbit_x}
The orbit of $x$ under the action by conjugation of $\U(3,2^2)$ is precisely $\cH$.
\end{lemma}
\begin{proof}
Since the $z_{i,j}$'s above are unitary and with the same minimal 
polynomial than $x$, it is clear that $\cH$ is contained in the orbit of $x$.

The stabilizer of $x$ consists of the diagonal matrices in $\U(3,2^2)$ ($27$ elements), while the size of $\U(3,2^2)$ is $36\times 6\times 3=648$. Hence the size of the orbit of $x$ is $648/27=24$, which coincides with the size of $\cH$.
\end{proof}

\begin{theorem}\label{th:PU32}
Let $\cO=\cO_{1,1}$ be the split Okubo algebra over a field $\FF$ of characteristic
not $3$ and containing a primitive cubic root of $1$. Then 
$\Aut(\Gamma_{1,1})$ is
isomorphic to the projective unitary group $\PU(3,2^2)$.
\end{theorem}
\begin{proof}
By Theorem \ref{th:special}, $\Aut(\Gamma_{1,1})$ is the semidirect product in
\eqref{eq:AutG11}, and this is independent of $\FF$ (as long as 
$\charac\FF\neq 3$ and $\omega\in\FF$). 
In particular, as the size of $\SL(2,3)$ is 
$\frac{8\times 6}{2}=24$, $\Aut(\Gamma_{1,1})$ has 
$9\times 24=216$ elements.

Lemma \ref{le:orbit_x} shows that $\PU(3,2^2)$ permutes the homogeneous components
of the grading $\Gamma_{1,1}$ over $\FF_4$, and hence $\PU(3,2^2)$ is contained
in $\Aut(\Gamma_{1,1})$. But $\PU(3,2^2)$ has $648/3=216$ elements, and the result follows.
\end{proof}

Theorem \ref{th:PU32} gives a surprising proof of the next result.

\begin{corollary}\label{co:PU32}
The projective unitary group $\PU(3,2^2)$ is, up to isomorphism, the semidirect 
product $\left(\ZZ/3\right)^2\rtimes \SL(2,3)$, and the projective special unitary group $\PSU(3,2^2)$ is the semidirect product of
$\left(\ZZ/3\right)^2$ and the quaternion group $Q_8$ (the Sylow
$2$-subgroup of $\SL(2,3)$).
\end{corollary}
\begin{proof}
The first part follows at once from Theorems \ref{th:PU32} and 
\ref{th:special}.

The classical group $\PSU(3,2^2)$ is a normal subgroup of index $3$ in 
$\PU(3,2^2)$,
and hence it contains its derived subgroup. It follows that, under
the isomorphism $\PU(3,2^2)\cong \left(\ZZ/3\right)^2\rtimes\SL(2,3)$, 
$\PSU(3,2^2)$ is the semidirect product $\left(\ZZ/3\right)^2\rtimes Q$, for an
order $8$ normal subgroup of $\SL(2,3)$. Then $Q$ is necessarily the unique 
(normal) Sylow $2$-subgroup of $\SL(2,3)$. The matrices
$\left(\begin{smallmatrix} 0&-1\\ 1&0\end{smallmatrix}\right)$ and
$\left(\begin{smallmatrix} 1&1\\ 1&-1\end{smallmatrix}\right)$ are
order $4$ elements of $\SL(2,3)$ that anticommute. Hence $Q$ is
generated by these two elements and it is isomorphic to the quaternion
group $Q_8$. The result follows.
\end{proof}

According to Corollary \ref{co:autos_not3}, the group $\PU(3,2^2)$ is (isomorphic to) the group of automorphisms of the Okubo algebra $\cO$ over the field of two elements. The group $\Aut(\cO)$ was shown in 
\cite[Propositions 7.10 and 7.13]{SV} to be naturally isomorphic to the group of automorphisms of the Okubo quasigroup $\mathsf{OK}(2)$ and its structure was determined, using \texttt{GAP}, to be an extension of $\SL(2,3)$ by 
$\left(\ZZ/3\right)^2$. By Corollary \ref{co:autos_not3}, $\Aut(\cO)$ is, up to isomorphism $\PU(3,2^2)$ and Corollary \ref{co:PU32} explains its structure.


\end{document}